\documentclass[12pt,leqno]{article}

\usepackage{amsfonts,url}
\usepackage{color}
\usepackage[margin=30mm]{geometry}
\usepackage{amssymb, amsmath, amsthm, amscd}
\usepackage{eucal,mathrsfs,dsfont}

\def\dist{\mathop{\rm dist}\nolimits}
\def\supp{\mathop{\rm supp}\nolimits}

\newcommand{\R}{\mathbb{R}}
\newcommand{\Rd}{ \mathbb{R}^{d}}
\newcommand{\N}{\mathbb{N}}
\newcommand{\Z}{\mathbb{Z}}

\newcommand{\indyk}[1]{\mathds{1}_{#1}}
\newcommand{\sfera}{ \mathds{S}}
\newcommand{\Borel}{ {\mathcal{B}}(\Rd) }
\newcommand{\Fourier}{ {\mathcal{F}}}

\newcommand{\nubounded}[1]{\bar{\nu}_{#1}}

\newcommand{\scalp}[2]{#1\cdot#2}

 \def\dist{\mathop{\rm
    dist}\nolimits} \def\diam{\mathop{\rm diam}\nolimits}

\newtheorem{lemat}{\indent\sc Lemma}
\newtheorem{prop}{\indent\sc Proposition}
\newtheorem{twierdzenie}{\indent\sc Theorem}
\newtheorem{wniosek}[lemat]{\indent\sc Corollary}

\newcounter{conum} 

\renewcommand{\Re}{\ensuremath{\operatorname{Re}}}

\begin{document}

\title{Estimates of transition densities and their derivatives for jump L\'evy processes}
\author{Kamil Kaleta, Pawe{\l} Sztonyk}
\footnotetext{ Kamil Kaleta \\ Institute of Mathematics,
  University of Warsaw,
  ul. Banacha 2,
  02-097 Warszawa, Poland,\\
  and \\ Institute of Mathematics and Computer Science,
  Wroc{\l}aw University of Technology,
  Wybrze{\.z}e Wyspia{\'n}\-skie\-go 27,
  50-370 Wroc{\l}aw, Poland.\\
  {\rm e-mail: KKaleta@mimuw.edu.pl, Kamil.Kaleta@pwr.wroc.pl} \\
}
\footnotetext{ Pawe{\l} Sztonyk \\ Institute of Mathematics and Computer Science,
  Wroc{\l}aw University of Technology,
  Wybrze{\.z}e Wyspia{\'n}\-skie\-go 27,
  50-370 Wroc{\l}aw, Poland.\\
  {\rm e-mail: Pawel.Sztonyk@pwr.wroc.pl} \\
}
\maketitle

\begin{center}
  Abstract
\end{center}
\begin{scriptsize}
  We give upper and lower estimates of densities of convolution semigroups of probability measures under explicit assumptions
  on the corresponding L\'evy measure and the L\'evy--Khinchin exponent. We obtain also estimates of derivatives of densities.
\end{scriptsize}

\footnotetext{2000 {\it MS Classification}:
Primary 60G51, 60E07; Secondary 60J35, 47D03, 60J45 .
K. Kaleta was supported by the National Science Center (Poland) internship
grant on the basis of the decision No. DEC-2012/04/S/ST1/00093. P. Sztonyk was supported by the National Science Center (Poland) grant on the basis of the decision No. DEC-2012/07/B/ST1/03356.
}
\section{Introduction}\label{Intro}

Let $d\in\{1,2,\dots\}$, $b\in\Rd$, and $\nu$ be a L\'evy measure on $\Rd$, i.e.,
$$
  \int_{\Rd} \left(1\wedge |y|^2\right)\,\nu(dy) < \infty.
$$ 

We always assume that $\nu(\Rd)=\infty$ and consider the convolution semigroup of probability measures $\{ P_t,\, t\geq 0 \}$ with the Fourier transform $\Fourier(P_t)(\xi)=\int_{\Rd} e^{i\scalp{\xi}{y}}P_t(dy)=\exp(-t\Phi(\xi))$, where
$$
  \Phi(\xi) =   - \int \left(e^{i\scalp{\xi}{y}}-1-i\scalp{\xi}{y}\indyk{B(0,1)}(y)\right)\nu(dy) - i\scalp{\xi}{b}  ,\quad \xi\in\Rd.
$$

There exists a L\'evy process $\{X_t,\,t\geq 0\}$
corresponding to the semigroup $\{P_t,\,t\geq 0\}$, i.e., each measure $P_t$ is the transition function of $X_t$.
For the rotation invariant $\alpha$-stable L\'evy processes we have $\nu(dy)=c|y|^{-d-\alpha}$ and $b=0$, where $\alpha\in (0,2)$. 
The asymptotic behaviour of its densities $p_t$
is well known (see, e.g., \cite{BG60}) and in this case we have $p_t(x)\approx \min(t^{-d/\alpha},t|x|^{-d-\alpha})$.
Explicit estimates for the first derivative of the transition 
density in this case are given in \cite[Lemma 5]{BJ07} and we have $|\nabla_x p(1,x)|\leq c|x|(1+|x|)^{-d-\alpha-2}$.

W.E. Pruitt and S.J. Taylor investigated in \cite{PT69} stable densities in the general setting, i.e., 
$\nu(dr d\theta) = r^{-1-\alpha}dr\mu(d\theta)$, where $\mu$ is a bounded measure on the unit sphere $\sfera$. 
They obtained the estimate $p_1(x)\leq c (1+|x|)^{-1-\alpha}$. Indeed, the upper bound can be attained if the spectral measure $\mu$ has an atom (see the estimates from below in \cite{H94} and \cite{H03}). P. G{\l}owacki and W. Hebisch proved in \cite{G93} and \cite{GH93} that if $\mu$
has a bounded density, $g_{\mu}$, with respect to the surface measure on $\sfera$ then $p_1(x)\leq c (1+|x|)^{-d-\alpha}$. When $g_{\mu}$ is continuous on $\sfera$ we even have
$\lim_{r\to\infty}r^{d+\alpha}p_1(r\theta)=cg_{\mu}(\theta)$, $\theta\in\sfera$ and if $g_{\mu}(\theta)=0$ then additionally $\lim_{r\to\infty}r^{d+2\alpha}p_1(r\theta)=c_{\theta} >0$, which was proved by J. Dziuba\'nski in \cite{D91}.

More recent asymptotic results for stable L\'evy processes are given 
in papers \cite{W07} and \cite{BS2007}. In particular if for some $\gamma\in[1,d]$ 
the measure $\nu$ is a $\gamma$--measure on $\sfera$ , i.e., 
$$
  \nu(B(x,r))\leq c r^{\gamma} \quad  \mbox{for every}\quad x \in\sfera,\, r\leq 1/2,
$$  
or equivalently 
$$
\mu(B(\theta,r)\cap \sfera)\leq c r^{\gamma-1},\quad  \theta\in\sfera,\, r\leq 1/2,
$$ 
then we have 
$$
  p_1(x)\leq c\, (1+|x|)^{-\alpha-\gamma},\quad x\in\Rd.
$$ 
Here and below we denote $B(x,r)=\{y\in\Rd:\:|y-x|<r\}$. By scaling $p_t(x)\leq c t^{-d/\alpha}(1+t^{-1/\alpha}|x|)^{-\alpha-\gamma}$ for every $t>0$. It follows also from 
\cite[Theorem 1.1]{W07} that if for some $\theta_0\in\sfera$ we have
$$
\mu(B(\theta_0,r)\cap \sfera)\geq c r^{\gamma-1},\quad  r\leq 1/2,
$$
then 
$$
  p_1(r\theta_0)\geq c\, (1+r)^{-\alpha-\gamma},\quad r>0.
$$

The estimates for more general L\'evy processes were next obtained in \cite{S10,S11,KnopKul,KSch1}. 
A recent paper \cite{BGR13} contains some estimates of densities for isotropic unimodal L\'evy processes 
with L\'evy-Khintchine exponents having the weak local scaling at infinity. Bounds for the transition density of a class of
Markov processes with jump intensities which are not necessarily translation
invariant but dominated by the L\'evy measure of the stable rotation invariant
process were given in \cite{ChKimKum,ChKum08,KSz}. Estimates for processes wich are solutions of some stochastic differential equations driven by L\'evy processes were given in \cite{Picard97}.

The main goal of the present paper is to extend the estimates in \cite{S10,S11} to more general class
of semigroups and processes. We want to emphasize that we consider an essentially wider class of L\'evy processes with 
L\'evy measures not necessarily absolutely continuous with respect to the underlying (e.g., Lebesgue) measure. 
We also include here processes with intensities of small jumps remarkably 
different than the stable one. The time-space asymptotics of the densities for this class of processes is still very little understood (see \cite{Mimica1,Mimica2}). The other novelty here are the estimates of the derivatives of the densities.

For a set $A\subset\Rd$ we denote $\delta(A)=\dist(0,A)=\inf\{|y|:\:y\in A\}$ and
$\diam(A)=\sup\{|y-x|:\:x,y\in A\}$. By ${\mathcal{B}}(\Rd )$ we denote Borel sets in $\Rd$.
We denote
$$
  \Psi(r)=\sup_{|\xi|\leq r} \Re \left(\Phi(\xi)\right),\quad r>0.
$$
We note that $\Psi$ is continuous and nondecreasing and $\sup_{r>0} \Psi(r)=\infty$, since $\nu(\Rd)=\infty$ (see \eqref{eq:Lm<Psi}). 
Let $\Psi^{-1}(s)=\sup\{r>0: \Psi(r)=s\}$ for
$s\in (0,\infty)$ so that $\Psi(\Psi^{-1}(s))=s$ for $s\in (0,\infty)$ and $\Psi^{-1}(\Psi(s))\geq s$ for $s>0$. Define
$$
  h(t)=\frac{1}{\Psi^{-1}\left(\frac{1}{t}\right)},\quad t>0.
$$
The function $h$ gives global estimates of densities of L\'evy processes (see Lemma \ref{lm:global_est_above}
and the metric defined in \cite{JKLSch2012}) and it appears also
in \cite{SchSW12} where gradient estimates of semigroups were proved.

The main results of the present paper are the following theorems.

\begin{twierdzenie}\label{th:main} Assume that
$\nu$ is a L\'evy measure such that $\nu(\Rd)=\infty$ and
\begin{equation}\label{eq:nu_estim}
  \nu(A) \leq M_1 f(\delta(A))[\diam(A)]^{\gamma},\quad A\in{\mathcal{B}}(\Rd ),
\end{equation}
where $\gamma\in[0,d]$, and $f:\:[0,\infty)\to [0,\infty]$ is nonincreasing function satisfying
\begin{equation}\label{eq:tech_assumpt}
  \int_{|y|>r} f\left(s\vee |y|-\frac{|y|}{2} \right) \,\nu(dy) 
  \leq 
  M_2 f(s) \Psi\left(\frac{1}{r} \right),\quad s>0,r>0,
\end{equation}
for some constants $M_1,M_2>0$.
We assume also that for a constant $M_3>0$ and a nonempty set $T\subseteq (0,\infty)$ we have
\begin{equation}\label{eq:Fourier_int}
  \int_{\Rd} e^{-t\Re\left(\Phi(\xi)\right)}|\xi|\, d\xi \leq M_3 \left(h(t)\right)^{-d-1},
  \quad t\in T.
\end{equation}

Then the measures $P_t$ are absolutely continuous with respect to the Lebesgue measure and there exist constants
$C_1,C_2,C_3$ such that their densities $p_t$ satisfy
\begin{eqnarray*}
  p_t(x+tb_{h(t)}) 
  & \leq & C_1 \left(h(t)\right)^{-d} \min\left\{ 1, t\left[h(t)\right]^{\gamma}
           f\left(|x|/4\right)
          + \, e^{-C_2 \frac{|x|}{h(t)}\log\left(1+\frac{C_3|x|}{h(t)}\right)}
            \right\},\\
   &     & x\in\Rd,\, t\in T,
\end{eqnarray*}
where
\begin{equation}\label{eq:def_br}
  b_r = \left\{
  \begin{array}{ccc}
    b - \int_{r<|y|<1} y \, \nu(dy) & \mbox{  if  } & r \leq 1,\\
    b + \int_{1<|y|<r} y \, \nu(dy) & \mbox{  if  } & r > 1.
  \end{array}\right.
\end{equation}
\end{twierdzenie}

We note that $T$ is an arbitrary subset of $(0,\infty)$ satisfying (\ref{eq:Fourier_int}). In particular
the Theorem \ref{th:main} can be applied either for small or for large times $t$. In Lemma \ref{lm:OdRene} below we give conditions
which yield (\ref{eq:Fourier_int}) for $T=(0,\infty)$.
All the assumptions of Theorem \ref{th:main} are satisfied by a wide class of semigroups 
and corresponding L\'evy 
processes, including stable, tempered stable, layered, relativistic, Lamperti and truncated
stable processes as well as geometric stable (for large times $t$) and some subordinated
processes. 
Some specific examples will be discussed in Section \ref{Examples}. 

The lower estimate
for symmetric L\'evy measures
is given in the following theorem.

\begin{twierdzenie}\label{th:p_est_below}
Assume that the L\'evy measure $\nu$ is symmetric, i.e. $\nu(D)=\nu(-D)$ for every 
$D\in\Borel$, $\nu(\Rd)=\infty$ and (\ref{eq:Fourier_int}) holds for a set $T\subset (0,\infty)$,
and there exists a constant $M_4>0$ such that
\begin{equation}\label{eq:nu_est_below}
  \nu(B(x,r)) \geq M_4 r^\gamma f(|x|+r),\quad x\in A,\, r>0,
\end{equation}
for some $A\in\Borel$, $\gamma\in [0,d]$ and a function $f:\: (0,\infty)\to [0,\infty)$.
Then there exist constants $C_4$, $C_5$ and $C_6>C_5$ 
such that
\begin{equation}\label{eq:p_est_below1}
  p_t(x+tb) \geq C_4 \left(h(t)\right)^{-d}\quad  \mbox{for}\quad |x|<C_6 h(t),\, t\in T,
\end{equation}
\begin{equation}\label{eq:p_est_below2}
  p_t(x+tb) \geq C_4 t\left[h(t)\right]^{\gamma-d} f\left(|x|+C_5 h(t)\right)\quad  
  \mbox{for}\quad  |x|\geq C_6 h(t),\, x\in A,\, t\in T.
\end{equation}
In particular
$$
  p_t(x+tb) \geq C_4 \left(h(t)\right)^{-d}  \min\left\{ 1, t\left[h(t)\right]^{\gamma}
           f\left(\min\{|x|+C_5 h(t),2|x|\}\right)\right\} ,\quad x\in A,\, t\in T.
$$
\end{twierdzenie}

The following estimate of derivatives is an extension of the results obtained
for stable processes in \cite{S10a}.

\begin{twierdzenie}\label{th:Derivatives}
If the L\'evy measure $\nu$ and a nonincreasing function $f$ satisfy
(\ref{eq:nu_estim}), (\ref{eq:tech_assumpt}), $\nu(\Rd)=\infty$ and there exist a constant $M_5>0$ and a set $T\subseteq (0,\infty)$ such that
\begin{equation}\label{eq:Fourier_derivatives}
   \int_{\Rd} e^{-t\Re\left(\Phi(\xi)\right)}|\xi|^m\, d\xi \leq M_5 \left(h(t)\right)^{-d-m},
  \quad t\in T,
\end{equation}
for some $m\in\N_0$, $m>\gamma$, 
then $p_t\in C^m_b(\Rd)$ and for every $n\in\N_0$ such that $m\geq n>\gamma$
and every $\beta\in\N^d_0$ such that $|\beta|\leq m-n$ there exists a constant $C_7=C_7(n,m)$ such that
\begin{eqnarray}\label{eq:der_est}
  |\partial^\beta_x p_t(x+tb_{h(t)})| 
  & \leq & C_7 \left(h(t)\right)^{-d-|\beta|} \min\left\{ 1,\,\, t\left[h(t)\right]^{\gamma}
           f\left(|x|/4\right)
          + \, \left(1+\frac{|x|}{h(t)}\right)^{-n}
            \right\},\nonumber \\
  &      & x\in\Rd,\, t\in T,
\end{eqnarray}
where $b_{h(t)}$ is given by (\ref{eq:def_br}).
\end{twierdzenie}

In Section \ref{Symbol} we give estimates of the real part of the characteristic exponent $\Re\Phi$ and the function $\Psi$ in terms of the L\'evy measure $\nu$ and we
consider sufficient conditions for assumptions (\ref{eq:Fourier_int}) and (\ref{eq:Fourier_derivatives}). We prove also that an inequality
opposite to (\ref{eq:Fourier_derivatives}) holds for every L\'evy measure.
In section \ref{Proof} we prove all the main theorems.
In Section \ref{Examples} we discuss examples. We focus on the specific type of L\'evy measures
$\nu$ such that $\nu(drd\theta)\approx r^{-1-\alpha}[\log(1+r^{-\kappa})]^{-\beta} dr\mu(d\theta)$ for suitable constants $\alpha,\kappa,\beta$
and a nondegenerate measure $\mu$ on the unit sphere $\sfera$.

We use $c,C,M$ (with subscripts) to denote finite positive constants
which depend only on $\nu$, $b$, and the dimension $d$. Any {\it additional} dependence
is explicitly indicated by writing, e.g., $c=c(n)$.
We write $f(x)\approx g(x)$ to indicate that there is a constant $c$ such that $c^{-1}f(x) \leq g(x) \leq c f(x)$.

\section{Estimates of characteristic exponent}\label{Symbol}

The characteristic exponent (symbol) $\Phi$ of the process is a continuous negative definite function and its basic properties are given, e.g., in \cite{J1}. In Proposition \ref{prop:1} we obtain both sides estimates for $\Re\Phi$ and 
$\Psi(r)=\sup_{|\xi|\leq r} \Re \left(\Phi(\xi)\right)$. 
We note that the estimates for $\Psi$ follow also from
combined results of \cite{Sch98}, Remark 4.8 and Section 3. of \cite{P81} but we include here a short direct proof (see also Lemma 6 in \cite{Grz2013}).

\begin{prop}\label{prop:1} Let 
 $$
   H(r)=\int 1\wedge \frac{|y|^2}{r^2}\,\nu(dy),\quad r>0.
 $$
 We have
 \begin{equation}\label{eq:PhiE}
   (1-\cos 1) \int_{|y|<1/|\xi|} |\scalp{\xi}{y}|^2 \nu(dy) \leq \Re(\Phi(\xi))\leq 2 H(1/|\xi|),\quad \xi\in\Rd\setminus\{0\},
 \end{equation}
 and there exists a constant $C_8$ such that
 \begin{equation}\label{eq:PsiE}
   C_8 H(1/r) \leq \Psi(r) \leq 2 H(1/r),\quad r>0.
 \end{equation}
\end{prop}
\begin{proof}
  We have
\begin{eqnarray*}
  \Re(\Phi(\xi)) 
  &   =  & \int (1-\cos(\scalp{\xi}{y}))\,\nu(dy) \\
  & \leq & \frac{1}{2} \int_{|y|\leq 1/|\xi|} |\scalp{\xi}{y}|^2 \,\nu(dy) + 2\int_{|y| > 1/|\xi|}\,\nu(dy) \\
  & \leq & \frac{1}{2} |\xi|^2 \int_{|y|\leq 1/|\xi|} |y|^2 \,\nu(dy)  + 2\int_{|y| > 1/|\xi|}\,\nu(dy) \\
  & \leq & 2 H(1/|\xi|),\quad \xi\in\Rd\setminus\{0\}.
\end{eqnarray*}
Let $f(s)=(1-\cos s)/s^2$. We have $f'(s)=(s\sin s - 2(1-\cos s))/s^3 $ and $g(s):=s\sin s - 2(1-\cos s)<0$ for $s\in (0,\pi)$,
since $g(0)=0$ and $g'(s)=s\cos s - \sin s <0$, for $s\in (0,\pi)$, hence
$f(s)$ is decreasing on $(0,1)$ and so $1-\cos s \geq (1-\cos 1) s^2,$ for $|s|\leq 1$. We obtain
$$
  \Re(\Phi(\xi))
  \geq (1-\cos 1) \int_{|y|<1/|\xi|} |\scalp{\xi}{y}|^2 \nu(dy),
$$
and (\ref{eq:PhiE}) follows.
The upper estimate in (\ref{eq:PsiE}) follows directly from (\ref{eq:PhiE}). For the lower estimate
we use the obvious inequality
$$
  \int_A g(x)\, dx \leq |A| \sup_{x\in A} g(x).
$$
Let $\delta\in (0,1)$ and $M_\delta = \bigcup_{k\in\Z} (\delta+k2\pi,2\pi-\delta+k2\pi)$, $c_1=(1-\cos\delta)/\delta^2$, $\kappa=3\delta/(2\pi-\delta)$ and let $\omega_0=1$ and $\omega_d=\pi^{d/2}/\Gamma(d/2+1)$ (the volume of the unit ball in $\Rd$).
We get
\begin{eqnarray*}
  \Psi(r) 
  &   =  & \sup_{|\xi|\leq r} \int (1-\cos(\scalp{\xi}{y}))\,\nu(dy) \\
  & \geq & \frac{1}{r^d \omega_d} \int_{|\xi|<r}\int (1-\cos(\scalp{\xi}{y}))\,\nu(dy) d\xi \\
  & \geq & \frac{1}{r^d \omega_d} \int_{|\xi|<r}\left( \int_{|\scalp{\xi}{y}|<\delta} (1-\cos(\scalp{\xi}{y})) \nu(dy) + 
           \int_{\scalp{\xi}{y} \in M_\delta} (1-\cos(\scalp{\xi}{y}))\,\nu(dy)\right) d\xi \\
  & \geq & \frac{1}{r^d \omega_d} \int_{|\xi|<r} \left( c_1 \int_{|\scalp{\xi}{y}|<\delta} |\scalp{\xi}{y}|^2 \nu(dy) + 
           c_1 \delta^2 \int_{\scalp{\xi}{y} \in M_\delta} \,\nu(dy)\right) d\xi \\
  &   =  & \frac{c_1}{r^d \omega_d} \left( \int \int_{|\scalp{\xi}{y}|<\delta,|\xi|<r} |\scalp{\xi}{y}|^2\,d\xi \nu(dy)
           + \delta^2 \int \int_{\scalp{\xi}{y}\in M_\delta,|\xi|<r}\,d\xi \nu(dy) \right) \\
  &   =  & \frac{c_1}{r^d \omega_d} \left( \int |y|^2 \int_{|\xi_1|<\frac{\delta}{|y|},|\xi|<r} \xi_1^2\,d\xi \nu(dy)
           + \delta^2 \int \int_{\xi_1\in \frac{M_\delta}{|y|},|\xi|<r}\,d\xi \nu(dy) \right) \\
  & \geq  & \frac{c_1}{r^d \omega_d} \left( \int_{|y|<\frac{\delta}{\kappa r}} |y|^2 \int_{|\xi|<\kappa r} \xi_1^2\,d\xi \nu(dy)
           + \delta^2 \int_{|y|\geq\frac{\delta}{\kappa r}} \int_{\xi_1\in \frac{M_\delta}{|y|},|\xi|<r}\,d\xi \nu(dy) \right) \\
  &   =   & \frac{c_1}{r^d \omega_d} \left( \int_{|y|<\frac{\delta}{\kappa r}} |y|^2 \frac{\omega_d}{d+2}(\kappa r)^{d+2} \nu(dy)
           + \delta^2 \int_{|y|\geq\frac{\delta}{\kappa r}} \int_{\xi_1\in \frac{M_\delta}{|y|},|\xi|<r}\,d\xi \nu(dy) \right). \\
\end{eqnarray*}
For $r|y|>\delta/\kappa=(2\pi-\delta)/3$ we have
$$
  \int_{\xi_1\not\in \frac{M_\delta}{|y|},|\xi|<r}\,d\xi \leq \omega_{d-1} r^{d-1} \frac{2\delta}{|y|}\left(2\left\lfloor \frac{r|y|+\delta}{2\pi} \right\rfloor +1\right) \leq  \omega_{d-1} r^d 2\kappa,
$$
since if $\left\lfloor \frac{r|y|+\delta}{2\pi}\right\rfloor\geq1$ then also $r|y|\geq2\pi-\delta$.
For $\delta$ such that $2\kappa\omega_{d-1}/\omega_d\leq 1/2$ this yields
$$
  \int_{\xi_1\in \frac{M_\delta}{|y|},|\xi|<r}\,d\xi \geq \frac{1}{2} \omega_d r^d,
$$
and we obtain
\begin{eqnarray*}
  \Psi(r) 
  & \geq & c_1 \left( \frac{\kappa^d}{d+2}  \int_{|y|<\frac{\delta}{\kappa r}}(\kappa r)^{2} |y|^2   \nu(dy) 
           + \frac{1}{2} \int_{|y|\geq\frac{\delta}{\kappa r}}\delta^2 \nu(dy) \right) \\
  &  \geq & \frac{c_1\kappa^d}{d+2} \int \left( |y|\kappa r \wedge  \delta\right)^2  \nu(dy) \geq \frac{c_1\kappa^{d+2}}{d+2} H(1/r).
\end{eqnarray*}
\end{proof}
\noindent

Now we prove the following technical lemma.

\begin{lemat}\label{lm:prawieTauber}
	Assume that for a function $f:\:(0,\infty)\to [0,\infty)$ exist a nonincreasing function $g:\: (0,\infty)\to [0,\infty)$ 
	and constants $m>0$, $a>\kappa\geq 0$, and
	$r_0\geq 0$	such that
  \begin{equation}\label{eq:tech_ass}
    \int_0^r s^a f(s)\, ds \leq m r^{\kappa}g(r)
  \end{equation}
  for every $r > r_0$. Then we have
  $$
    \int_r^\infty f(s)\, ds \leq \frac{ma}{a-\kappa} r^{\kappa-a}g(r),
  $$
  for every $r>r_0$.
\end{lemat}
\begin{proof}
By (\ref{eq:tech_ass}) for every $r>r_0$ we have
$$
  \int_r^\infty \frac{1}{t^{a+1}}   \left( \int_0^t s^a f(s)\, ds\right)\, dt 
 \leq m \int_r^\infty  t^{\kappa-a-1} g(t) \, dt \leq m g(r) \int_r^\infty  t^{\kappa-a-1} \, dt 
 = \frac{m}{a-\kappa}r^{\kappa-a}g(r)
$$
Furthermore, changing the order of integration we obtain
\begin{eqnarray*}
  \int_r^\infty \frac{1}{t^{a+1}}   \left( \int_0^t s^a f(s)\, ds\right) dt 
  &  =   & \int_0^r s^a f(s) \int_r^\infty \frac{1}{t^{a+1}} \, dt \, ds+ \int_r^\infty s^a f(s) \int_s^\infty \frac{1}{t^{a+1}} \, dt\, ds \\
  &  =   & \frac{1}{ar^a} \int_0^r s^a f(s) \, ds + \frac{1}{a} \int_r^\infty f(s)\, ds \\
  & \geq & \frac{1}{a} \int_r^\infty f(s)\, ds,
\end{eqnarray*}
and the lemma follows.

\end{proof}

The following corollaries which give estimates of $\Re\Phi$ for the more specific case of L\'evy measure
follow directly from Proposition \ref{prop:1} and Lemma \ref{lm:prawieTauber}.

\begin{wniosek}\label{l:RePhi>}
  If $\mu$ is nondegenerate, i.e., the support of $\mu$ is not contained in any proper linear subspace
of $\Rd$, and
  $$
    \nu(A) \geq M_6 \int_\sfera \int_0^\infty \indyk{A}(s\theta) f(s)\,  ds\mu(d\theta),
  $$ 
  where $f:\: (0,\infty)\to [0,\infty)$, 
  then
  $$
    \Re(\Phi(\xi)) \geq C_9 |\xi|^2 g_1(1/|\xi|),
  $$
  where $g_1(r)=\int_0^r s^2 f(s)\, ds$.
\end{wniosek}

\begin{wniosek}\label{l:RePhi<}
  Let $f:\: (0,\infty)\to [0,\infty)$ be such that
  $$
    \nu(A) \leq M_{7} \int_\sfera \int_0^\infty \indyk{A}(s\theta) f(s)\,  ds\mu(d\theta),
  $$
  and
  $$
    \int_0^r s^2 f(s)\, ds \leq M_{8} r^{\kappa}g_2(r),\quad r>0, 
  $$
  for constants $M_{7}, M_{8}>0$, $2>\kappa\geq 0$ and nonincreasing function $g_2:\: (0,\infty)\to [0,\infty)$.
  Then there exists a constant $C_{10}$ such that
  $$
    \Re(\Phi(\xi)) \leq C_{10} |\xi|^{2-\kappa} g_2(1/|\xi|),\quad \xi\in\Rd\setminus\{0\}.
  $$
\end{wniosek}

In the following lemma we will prove that the inequality opposite to (\ref{eq:Fourier_derivatives}) 
holds for every L\'evy measure.

\begin{lemat}\label{lm:Oppos}
  If $\nu(\Rd)=\infty$ then for every $m\in\N_0$ there exists a constant $C_{11}=C_{11}(m)$ such that
  $$
    \int_{\Rd} e^{-t\Re\left(\Phi(\xi)\right)}|\xi|^m \, d\xi \geq C_{11} \left(h(t)\right)^{-d-m},
  $$ 
  for every $t>0$.
\end{lemat}
\begin{proof} Using the fact that $\Psi$ is increasing and $\Psi(\Psi^{-1}(s))=s$ for every $s>0$, we get
  \begin{eqnarray*}
    \int_{\Rd} e^{-t\Re\left(\Phi(\xi)\right)}|\xi|^m \, d\xi 
    & \geq & \int_{\Rd} e^{-t\Psi(|\xi|)}|\xi|^m \, d\xi \\
    &   =  & c_1 \int_0^\infty e^{-t\Psi(s)} s^{m+d-1}\, ds \\
    & \geq & c_1 \int_0^{\Psi^{-1}(1/t)} e^{-t\Psi(s)} s^{m+d-1}\, ds \\
    & \geq & \frac{c_1 e^{-1}}{m+d} \left[\Psi^{-1}(1/t)\right]^{m+d} \\
    &   =  & \frac{c_1 e^{-1}}{m+d} \left[h(t) \right]^{-m-d}.
  \end{eqnarray*}
\end{proof}

 Now we give conditions which guarantee that the assumptions (\ref{eq:Fourier_int})
 and (\ref{eq:Fourier_derivatives}) hold.

\begin{lemat}\label{lm:OdRene}
    Assume that there is a strictly increasing function 
    $F: [0,\infty)\to [0,\infty)$ such that 
    $F(0)=0$, $\lim_{s\to \infty} F(s) = \infty$,
    which is differentiable and which satisfies
    \begin{equation}\label{eq:doublingF}
        F^{-1}(2 s)\leq M_{9} F^{-1}(s),\quad s>0,
    \end{equation}
    and
    $$
       M_{10}^{-1} F(|\xi|) \leq \Re\Phi(\xi) \leq M_{10} F(|\xi|), \quad \xi\in\Rd,
    $$
    for some constants $M_{9}, M_{10}$.
    Then there exists a constant $C_{12}=C_{12}(m)$ such that
    \begin{equation}\label{eq:22}
   C^{-1}_{12}  \left(h(t)\right)^{-d-m} \leq \int_{\Rd} e^{-t\Re\left(\Phi(\xi)\right)}|\xi|^m\, d\xi 
   \leq C_{12}  \left(h(t)\right)^{-d-m},
  \quad t>0.
  \end{equation}
  for every $m\in\N\cup\{0\}$. 
 \end{lemat}
 
 \begin{proof}
   We follow the argumentation given in \cite{SchSW12}, proof of Theorem 1.3.
   We denote $g(s) = F^{-1}(s)$. We have
   $$
     \Psi(s) = \sup_{|x|<s} \Re \Phi(\xi) \approx \sup_{|x|<s} F(|\xi|) = F(s),
   $$
   and this yields
   $$
     \Psi^{-1}(c_1 s) \leq F^{-1}(s) \leq \Psi^{-1}(c_2 s),\quad s>0.
   $$
   It follows from (\ref{eq:doublingF}) that $g(2^n s) \leq c_3^n g(s)$ for $c_3=M_{9}$ 
   and for $c_4=M_{10}^{-1}$ this yields
   \begin{eqnarray*}
     \int e^{-t\Re \Phi(\xi)}|\xi|^m\, d\xi 
     & \leq & \int e^{-tc_4 F(|\xi|)}|\xi|^m\, d\xi \\
     &   =  & c_5 \int_0^\infty e^{-tc_4 F(s)}s^{m+d-1}\, ds \\
     &   =  & c_5 \int_0^\infty e^{-c_4 u/c_2} \left[ g\left(\frac{u}{c_2t}\right) \right]^{m+d-1}
              g'\left(\frac{u}{c_2t}\right)\frac{1}{c_2t}\, du \\
     &  =   & c_5 \left(\int_0^1 + \int_1^\infty\right) \\
     & \leq & \frac{c_5}{m+d} \left[g\left(\frac{1}{c_2t}\right)\right]^{m+d} \\
     &      &  + c_5 \sum_{n=1}^\infty \int_{2^{n-1}}^{2^n}
              e^{-c_4 u/c_2} \left[ g\left(\frac{u}{c_2t}\right) \right]^{m+d-1} 
              g'\left(\frac{u}{c_2t}\right)\frac{1}{c_2t}\, du \\
     & \leq & \frac{c_5}{m+d}\left( \left[g\left(\frac{1}{c_2t}\right)\right]^{m+d} 
                +  \sum_{n=1}^\infty e^{-c_4 2^{n-1}/c_2} \left[g\left(\frac{2^n}{c_2t}\right) \right]^{m+d}\right) \\
     & \leq & \frac{c_5}{m+d}\left( \left[g\left(\frac{1}{c_2t}\right)\right]^{m+d} + 
              \left[g\left(\frac{1}{c_2t}\right) \right]^{m+d} \sum_{n=1}^\infty 
            e^{-c_42^{n-1}/c_2} c_3^{n(m+d)}\right) \\
     &   =  & c_6 \left[g\left(\frac{1}{c_2t}\right)\right]^{m+d} \leq c_6 \left[\Psi^{-1}(1/t) \right]^{m+d} = c_6 \left[h(t) \right]^{-m-d}.
   \end{eqnarray*}
   The estimate from below in (\ref{eq:22}) follows from Lemma \ref{lm:Oppos}.
 \end{proof}

\section{Proof of theorems}\label{Proof}

We will now prove the theorems. In the following we often assume that (\ref{eq:Fourier_int}) is
satisfied which gives the existence of densities $p_t\in C^1_b(\Rd)$ of $P_t$
for $t\in T$. 
We note that several necessary and sufficient conditions for the existence
of (smooth) transition probability densities for L\'evy processes and isotropic L\'evy
processes are are given in \cite{KSch}.

In the following two lemmas we obtain estimates of $p_t$ by constants depending on $t$.

\begin{lemat}\label{lm:global_est_above}
  If $\nu(\Rd)=\infty$ and (\ref{eq:Fourier_int}) holds then there exists a constant $C_{13}$ such that
  $$ 
    p_t(x) \leq C_{13} \left(h(t)\right)^{-d},\quad t\in T.
  $$
\end{lemat}
\begin{proof}
  We have
  \begin{eqnarray*}
    p_t(x)
    &   =  & \left(2\pi\right)^{-d} \int e^{-i\scalp{x}{\xi}} e^{-t\Phi(\xi)}\, d\xi
     \leq  \left(2\pi\right)^{-d} \int e^{-t\Re(\Phi(\xi))}\, d\xi \\
    &   =  & \left(2\pi\right)^{-d} \left(\int_{|\xi|\leq (1/h(t))} e^{-t\Re(\Phi(\xi))}\, d\xi 
              + \int_{|\xi| > (1/h(t))} e^{-t\Re(\Phi(\xi))}\, d\xi \right)\\
    & \leq & \left(2\pi\right)^{-d} \left( c_1 \left(h(t)\right)^{-d} 
             + h(t) \int e^{-t\Re(\Phi(\xi))}|\xi| \, d\xi \right) \\
    & \leq & c_2 \left(h(t)\right)^{-d},
  \end{eqnarray*}
  for $t\in T$. Here we use (\ref{eq:Fourier_int}) in the last inequality above.
\end{proof}

\begin{lemat}\label{lm:est_below_0}
If $\nu(\Rd)=\infty$ and (\ref{eq:Fourier_int}) holds and $\nu$ is symmetric 
then there exist constants 
$C_{14},C_6$ such that
\begin{equation}\label{eq:est_by0_below}
  p_t(x+tb) \geq C_{14} \left(h(t)\right)^{-d},\quad t\in T,\, |x|\leq C_6 h(t).
\end{equation}
\end{lemat}
\begin{proof}
  It follows from Lemma \ref{lm:Oppos} and the symmetry of $\nu$ that
  \begin{eqnarray*}
    p_t(tb)
    &   =  & \left(2\pi\right)^{-d} \int e^{-t\Phi(\xi)}e^{-it\scalp{b}{\xi} }\, d\xi 
             = \left(2\pi\right)^{-d} \int e^{-t\Re(\Phi(\xi))}\, d\xi\\
    & \geq & c_1 \left(h(t)\right)^{-d}.
  \end{eqnarray*}
  For every $j\in\{1,...,d\}$ and $t\in T$, by (\ref{eq:Fourier_int}) we get
  \begin{eqnarray*}
  \left|\frac{\partial p_t}{\partial y_j}(y)\right|
  &   =  & \left|(2\pi)^{-d}\int_{\Rd}
           (-i)\xi_j e^{-i\scalp{y}{\xi}}
           e^{-t\Phi(\xi)}\, d\xi\right| \\
  & \leq & \left(2\pi\right)^{-d} \left(\int_{|\xi|\leq (1/h(t))} e^{-t\Re(\Phi(\xi))}|\xi|\, d\xi 
              + \int_{|\xi| > (1/h(t))} e^{-t\Re(\Phi(\xi))}|\xi|\, d\xi \right)\\
    & \leq & \left(2\pi\right)^{-d} \left( c_2 \left(h(t)\right)^{-d-1} 
             +  \int e^{-t\Re(\Phi(\xi))}|\xi| \, d\xi \right) \\
    & \leq & c_3 \left(h(t)\right)^{-d-1},\quad y\in\Rd.
\end{eqnarray*}
It follows that
$$
  p_t(x+tb) \geq c_1 \left(h(t)\right)^{-d} - dc_3 \left(h(t)\right)^{-d-1} |x| 
  \geq \frac{c_1}{2} \left(h(t)\right)^{-d} ,
$$
provided that $|x|\leq \frac{c_1}{2dc_3}h(t)$, which clearly yields (\ref{eq:est_by0_below}).
\end{proof}

For $r>0$ we denote $\tilde{\nu}_r(dy)=\indyk{B(0,r)}(y)\nu(dy)$. 
We consider the semigroup of measures $\{\tilde{P}^r_t,\; t\geq 0\}$ such that
$$
  \Fourier(\tilde{P}^r_t)(\xi) 
    =     \exp\left(t \int \left(e^{i\scalp{\xi}{y}}-1-i\scalp{\xi}{y}\right)
            \tilde{\nu}_r(dy)\right), \quad \xi\in\Rd.\,
$$
We have
\begin{eqnarray}\label{eq:FTildePEstimate}
  |\Fourier(\tilde{P}^r_t)(\xi)| 
  &   =  & \exp\left(-t\int_{|y|<r}
           (1-\cos(\scalp{y}{\xi}))\,
           \nu(dy)\right) \nonumber \\
  &   =  & \exp\left(-t\left(\Re(\Phi(\xi))-\int_{|y|\geq r}
           (1-\cos(\scalp{y}{\xi}))\,
           \nu(dy)\right)\right) \nonumber \\
  & \leq & \exp(-t\Re(\Phi(\xi)))\exp(2t\nu(B(0,r)^c)),
           \quad \xi\in\Rd.
\end{eqnarray}
It follows that $\int |\Fourier(\tilde{P}^r_t)(\xi)|\cdot |\xi|\, d\xi \leq e^{2t\nu(B(0,r)^c)} 
\int e^{-t\Re\left(\Phi(\xi)\right)}|\xi|\, d\xi$, hence if (\ref{eq:Fourier_int}) holds then for every $r>0$ and $t\in T$ the measure $\tilde{P}^r_t$ is absolutely continuous with respect to
the Lebesgue measure with density, say, $\tilde{p}^r_t\in C^1_b(\Rd)$ (see, e.g., Proposition 2.5 in \cite{Sato}).

We will often use $\tilde{P}^r_t$ and $\tilde{p}^r_t$ with $r=h(t)$ and
for simplification we will denote 
$$ 
  \tilde{P}_t=\tilde{P}^{h(t)}_t\,\, \mbox{and}\,\, \tilde{p}_t=\tilde{p}^{h(t)}_t.
$$

We note also that there exists a constant $M_0$ such that
\begin{equation}\label{eq:Lm<Psi}
  \nu(B(0,r)^c)\leq M_0\, \Psi(1/r),\quad r>0,
\end{equation} 
which follows from Proposition \ref{prop:1} (see also \cite{SchSW12}, the proof of Proposition 2.2, Step 3).

Using (\ref{eq:FTildePEstimate}) and (\ref{eq:Lm<Psi}) we obtain
\begin{eqnarray}\label{eq:FTildePhest}
  |\Fourier(\tilde{P}_t)(\xi)| 
  & \leq & \exp(-t\Re(\Phi(\xi)))\exp(2t\nu(B(0,h(t))^c)) \nonumber \\
  & \leq & \exp(-t\Re(\Phi(\xi)))\exp(2tM_0\Psi(1/h(t))) \nonumber \\
  &  =   & \exp(-t\Re(\Phi(\xi)))\exp(2M_0), \quad \xi\in\Rd, t\in T,
\end{eqnarray}
since $\Psi(1/h(t))=1/t$.

The L\'evy measures with bounded support are discussed, e.g., in Section 26 of \cite{Sato}, where estimates
of tails of corresponding distributions are included. We extended these results in
\cite{S11} to estimates of densities and in the following lemma we use the results of \cite{S11} in our new more general context.
\begin{lemat}\label{lm:small_jumps_est}
  If $\nu(\Rd)=\infty$ and (\ref{eq:Fourier_int}) holds then there exist constant $C_{15},C_{16}$ and $C_{17}$ such that
  \begin{equation}\label{eq:small_jumps_est}
    \tilde{p}_t(x)\leq C_{15} \left[h(t) \right]^{-d} 
    \exp\left[  \frac{-C_{16}|x|}{h(t)}\log\left(1+\frac{C_{17}|x|}{h(t)}\right)\right],\quad x\in\Rd, t\in T.
  \end{equation}
\end{lemat}

\begin{proof}
  Let $g_t(y)=[h(t)]^d\tilde{p}_t\left(h(t)y\right)$. We consider the infinitely divisible distribution $\pi_t(dy)=g_t(y)\, dy$. We note that
\begin{eqnarray*}
  \Fourier(\pi_t)(\xi) 
  &  =   &  \exp\left(t 
            \int \left(e^{i\scalp{\xi(h(t))^{-1}}{y}}-1-
            i\scalp{\xi(h(t))^{-1}}{y}\indyk{B(0,h(t))}(y)\right)
            \tilde{\nu}_{h(t)}(dy)\right)\\
  &  =   & \exp\left( 
           \int \left(e^{i\scalp{\xi}{y}}-1-
           i\scalp{\xi}{y}\indyk{B(0,1)}(y)\right)
           \lambda_t (dy)\right),  \quad \xi\in\Rd,
\end{eqnarray*}
where $\lambda_t(A)=t\tilde{\nu}_{h(t)}(h(t)A)$ is the L\'evy measure of $\pi_t$.

From (\ref{eq:FTildePhest}) and (\ref{eq:Fourier_int})
for every $j\in\{1,\dots,d\}$ and $t\in T$ we obtain
\begin{eqnarray*}
  \left|\frac{\partial g_t}{\partial y_j}(y)\right|
  &   =  & [h(t)]^{d+1} \left|(2\pi)^{-d}\int_{\R}
           (-i)\xi_j e^{-i\scalp{h(t)y}{\xi}}
           \Fourier(\tilde{p}_t)(\xi)d\xi\right|  \\
  & \leq & [h(t)]^{d+1} (2\pi)^{-d} \int |\xi| 
           e^{2M_0} e^{-t\Re(\Phi(\xi))} \,d\xi  \\
  & \leq & c_1.
\end{eqnarray*}

Similarly we get 
\begin{eqnarray}\label{eq:gbounded}
  g_t(y)
  &   =  & [h(t)]^{d} (2\pi)^{-d}\int
           e^{-i\scalp{h(t)y}{\xi}}
           \Fourier(\tilde{p}_t)(\xi)\, d\xi  \nonumber \\
  & \leq & [h(t)]^{d} (2\pi)^{-d} \int e^{2M_0} e^{-t\Re(\Phi(\xi))} \,d\xi \nonumber \\
  &   =  & [h(t)]^{d} (2\pi)^{-d} e^{2M_0} \left[\int_{|\xi|\leq (1/h(t))}
           e^{-t\Re(\Phi(\xi))} \,d\xi + \int_{|\xi|> (1/h(t))} 
           e^{-t\Re(\Phi(\xi))} \,d\xi\right] \nonumber \\
  & \leq & [h(t)]^{d} (2\pi)^{-d} e^{2M_0} \left[ c_2 [h(t)]^{-d} + 
           h(t) \int_{|\xi|> (1/h(t))} |\xi| 
           e^{-t\Re(\Phi(\xi))} \,d\xi\right] \nonumber \\
  & \leq & c_3.
\end{eqnarray}
It follows from (2.16) in \cite{SchSW12} that
$$
  \int |y|^2\, \lambda_t(dy) 
  =   t \int (|y|/h(t))^2\,\tilde{\nu}_{h(t)}(dy) \leq    c_4.
$$
We have also
\begin{eqnarray*}
  \int_{|y|>1} y_j \lambda_t(dy) 
  &  =   &  t (h(t))^{-1}\int_{B(0,h(t))^c} y_j \tilde{\nu}_{h(t)}(dy) = 0. 
\end{eqnarray*}
It follows from Lemma 2 in \cite{S11} and (\ref{eq:gbounded}) that
$$
  g_t(y) \leq c_5 \exp \left( -c_6|y| \log\left( c_7|y|\right)\right) \leq c_8 \exp \left( -c_9|y| \log\left(1+ c_{10}|y|\right)\right),
$$
for $y\in\Rd$, and this yields
$$
  \tilde{p}_t(x) \leq c_8 (h(t))^{-d} \exp\left( -\frac{c_9|x|}{h(t)} \log\left(1+ \frac{c_{10}|x|}{h(t)}\right)\right),
$$
for $x\in\Rd$, $t\in T$.
\end{proof}

For $r>0$ we denote $\nubounded{r}(dy)= \indyk{B(0,r)^c}(y)\,\nu(dy)$ and consider 
the probability measures $\{\bar{P}^r_t,\; t\geq 0\}$ such that
\begin{equation}\label{eq:FPbar}
  \Fourier(\bar{P}^r_t)(\xi) =
  \exp\left(t \int (e^{i\scalp{\xi}{y}}-1)\,
  \nubounded{r}(dy)\right)\, ,
  \quad \xi\in\Rd\, .
\end{equation}
Note that
\begin{eqnarray}\label{eq:exp}
  \bar{P}^r_t
  &  =  & \exp(t(\nubounded{r}-|\nubounded{r}|\delta_0)) =  \sum_{n=0}^\infty \frac{t^n\left(\nubounded{r}-|\nubounded{r}|\delta_0)\right)^{n*}}{n!} \\
  &  =  & e^{-t|\nubounded{r}|} \sum_{n=0}^\infty \frac{t^n\nubounded{r}^{n*}}{n!}\,,\quad t\geq 0\, . \nonumber
\end{eqnarray}

\begin{lemat}
If $\nu$ is a L\'evy measure and $f:\:[0,\infty)\to (0,\infty]$ is nonincreasing 
function satisfying (\ref{eq:nu_estim}) and if for some $r>0$ we have
\begin{equation}\label{eq:pomocn}
  \int_{|y|>r} f\left(s\vee |y|-\frac{|y|}{2} \right) \,\nu(dy) 
  \leq 
  M_2 f(s) \Psi\left(\frac{1}{r} \right),\quad s>0,
\end{equation}
with a constant $M_2$, then
  \begin{equation}\label{eq:nu_n*_est}
    \nubounded{r}^{n*}(A) \leq C_{18}^n \left[\Psi(1/r)\right]^{n-1} 
    f\left(\delta(A)/2\right)\left[\diam(A)\right]^{\gamma},
  \end{equation}
for $n\in\N$ and $A\in\Borel$ such that $\delta(A)>0$, $\diam(A)<\infty$, with a constant $C_{18}:= \max\{M_0,M_1+M_2\}$.
\end{lemat}

\begin{proof}
We use induction. For $n=1$ the lemma follows from (\ref{eq:nu_estim}). Let (\ref{eq:nu_n*_est}) hold for some $n\in\N$ and constant $c_0=C_{18}$ and $A$ be a set such that $\delta(A)>0$. For $y\in\Rd$ we denote $D_y=\{z\in\Rd:\: |z| > \frac{1}{2}|z+y|\}=
\left(\overline{B\left(\frac{1}{3}y,\frac{2}{3}|y|\right)}\right)^c$. We have
\begin{eqnarray*}
  \nubounded{r}^{(n+1)*}(A) 
  &  =  & \int \nubounded{r}(A-y)\, \nubounded{r}^{n*}(dy) \\
  &  =  & \int \nubounded{r}\left((A-y)\cap D_y\right)\, \nubounded{r}^{n*}(dy) + \int\nubounded{r}\left((A-y)\cap D_y^c\right)\, \nubounded{r}^{n*}(dy) \\
  &  =  & I + II.
\end{eqnarray*}

We note that for $z\in (A-y)\cap D_y$ we have $z+y\in A$ and $|z|>\frac{1}{2} |z+y|$, therefore $|z| >\frac{1}{2} \delta(A)$ and 
$\delta((A-y)\cap D_y)>\frac{1}{2} \delta(A)$. Furthermore, $\diam((A-y)\cap D_y) \leq  \diam(A)$ and using (\ref{eq:nu_estim}) and (\ref{eq:Lm<Psi}) 
we obtain
\begin{eqnarray*}
  I
  & \leq & M_1 f\left(\delta(A)/2\right) \left(\diam(A)\right)^\gamma |\nubounded{r}^{n*}| \\
  & \leq & M_1 M_0^n \left(\Psi(1/r)\right)^n f\left(\delta(A)/2\right)
           \left(\diam(A)\right)^\gamma.
\end{eqnarray*}

We have
\begin{eqnarray*}
  II
  &  =   & \int\int \indyk{A-y}(z) \indyk{D_y^c}(z)\, \nubounded{r}(dz)\nubounded{r}^{n*}(dy) \\
  &  =   & \int\int \indyk{A-z}(y)\indyk{B(-z,2|z|)^c}(y) \, \nubounded{r}^{n*}(dy)\nubounded{r}(dz)\\
  &  =   & \int \nubounded{r}^{n*}\left((A-z)\cap B(-z,2|z|)^c\right)\,\nubounded{r}(dz),
\end{eqnarray*}

Let $y\in V_z :=(A-z)\cap B(-z,2|z|)^c$. We then have $y+z\in A$, so $|y+z|\geq \delta(A)$,
and $|y+z|\geq 2|z|$. Furthermore $|y|\geq |y+z|-|z|$ and this yields 
$$
  \delta(V_z) \geq \inf_{y\in V_z} |y+z| - |z| \geq \left(\delta(A)\vee 2|z|\right)-|z| \geq \frac{1}{2}\delta(A),
$$
and by (\ref{eq:pomocn}) and the induction hypothesis we get
\begin{eqnarray*}
  II
  & \leq & c_0^n \left(\Psi(1/r)\right)^{n-1} \left(\diam(A)\right)^\gamma \int f\left(\frac{ \left(\delta(A)\vee 2|z|\right)-|z|}{2}\right) \, \nubounded{r}(dz) \\
  & \leq & c_0^n \left(\Psi(1/r)\right)^{n-1} \left(\diam(A)\right)^\gamma M_2 
           f\left(\delta(A)/2 \right) \Psi(1/r) \\
  & =    & M_2 c_0^n \left(\Psi(1/r)\right)^{n} 
           f\left(\delta(A)/2 \right) \left(\diam(A)\right)^\gamma .
\end{eqnarray*}
Indeed, we see that the lemma follows by taking $c_0 := \max\{M_0,M_1+M_2\}$.
\end{proof}

\begin{wniosek}\label{cor:nub_est}
  If (\ref{eq:nu_estim}) and (\ref{eq:tech_assumpt}) hold then 
  $$
    \nubounded{r}^{n*}(B(x,\rho)) \leq 
    C_{18}^n \left[\Psi(1/r)\right]^{n-1} f\left(|x|/4\right) (2\rho)^{\gamma},
  $$
  for every $x\in\Rd\setminus\{0 \}$, $\rho<|x|/2$ and $r>0$, $n\in\N$.
\end{wniosek}

\begin{proof}[Proof of Theorem $\ref{th:main}$]
We have
\begin{displaymath}
  P_t=\tilde{P}^r_t \ast \bar{P}^r_t \ast \delta_{t b_r}\,,\quad t\geq 0,
\end{displaymath}
where $\bar{P}^r_t$ is defined by (\ref{eq:FPbar}) and $b_r$ by (\ref{eq:def_br}).
Of course
$$
  p_t= \tilde{p}^r_t * \bar{P}^r_t\ast \delta_{t b_r} \,, \quad t\in T.
$$

We will denote 
$$
  \bar{P}_t=\bar{P}_t^{h(t)}.
$$

We have $\Psi(1/h(t))=1/t$ and it follows from Corollary \ref{cor:nub_est} and (\ref{eq:exp}) that
\begin{equation}\label{eq:bar_P_schonwieder}
  \bar{P}_t(B(x,\rho)) \leq c_1 t f\left(|x|/4\right) \rho^\gamma,
\end{equation}
for $\rho\leq \frac{1}{2}|x|$ and $t>0$.

We denote
$$
  g(s)=e^{-C_{16}s\log(1+C_{17}s)},\quad s\geq 0,
$$
where constants $C_{16},C_{17}$ are given by (\ref{eq:small_jumps_est}). We note that $g$ is decreasing, continuous on $[0,\infty)$ and $g(s)\leq c_2 s^{-2\gamma}$, for some
$c_2>0$, which yields that the inverse function $g^{-1}:\: (0,1]\to [0,\infty)$ exists, is decreasing, and $g^{-1}(s)\leq \left(c_2/s\right)^{1/(2\gamma)}$.
In particular
$$
  \int_0^1 \left(g^{-1}(s)\right)^{\gamma}\,ds < \infty.
$$
Using Lemma \ref{lm:small_jumps_est} and (\ref{eq:bar_P_schonwieder}) we obtain
\begin{eqnarray*}
  \tilde{p}_t * \bar{P}_t (x)
  &  =   & \int \tilde{p}_t(x-y)\, \bar{P}_t(dy) \\
  & \leq & \int C_{16} [h(t)]^{-d} g(|x-y|/h(t)) \, \bar{P}_t(dy) \\
  &   =  & C_{16} [h(t)]^{-d} \int \int_0^{g(|x-y|/h(t))} \, ds\, \bar{P}_t(dy) \\
  &   =  & C_{16} [h(t)]^{-d} \int_0^1 \int \indyk{\{y\in\Rd:\: g(|x-y|/h(t)) >s\}} \, \bar{P}_t(dy) ds \\
  &   =  & C_{16} [h(t)]^{-d} \int_0^1 \bar{P}_t\left(B(x,h(t)g^{-1}(s))\right) ds \\
  & \leq & c_1C_{16} [h(t)]^{-d} \left(\int_{g(\frac{|x|}{2h(t)})}^1 t f\left(|x|/4\right) 
           \left(h(t)g^{-1}(s)\right)^\gamma\, ds + \int_0^{g(\frac{|x|}{2h(t)})}\, ds\right) \\
  & \leq & c_1C_{16} [h(t)]^{-d} \left( t[h(t)]^{\gamma} f\left(|x|/4\right) \int_0^1  
           \left(g^{-1}(s)\right)^\gamma\, ds + g\left(\frac{|x|}{2h(t)}\right) \right) \\
  &  =   & c_3 [h(t)]^{-d} \left( t[h(t)]^{\gamma} f\left(|x|/4\right) 
           + g\left(\frac{|x|}{2h(t)}\right) \right). \\
\end{eqnarray*}

This and Lemma \ref{lm:global_est_above} yield
\begin{eqnarray*}
  p_t(x+t b_{h(t)}) 
  &  =   & \int \tilde{p}_t * \bar{P}_t (x+t b_{h(t)}-y) \delta_{t b_{h(t)}}(dy) \\
  &  =   & \tilde{p}_t * \bar{P}_t (x) \\
  & \leq & c_4 [h(t)]^{-d} \min\left\{1, t[h(t)]^{\gamma}
          f\left(|x|/4\right) + \,g\left(\frac{|x|}{2h(t)}\right) \right\},
\end{eqnarray*}
for $t\in T$.
\end{proof}

The following Lemma which will be used in the proof of Theorem \ref{th:p_est_below} 
was communicated to us by Tomasz Grzywny.

\begin{lemat}\label{lm:OdTomkaG} If $\nu(\Rd)=\infty$ then we have
$$
  \lim_{a\to 0^+}\sup_{t>0} \frac{h(at)}{h(t)} =0.
$$
\end{lemat}
\begin{proof}
   Since $\nu(\Rd)=\infty$, the function $H(r)=\int (1\wedge (|y|^2/r^2))\,\nu(dy)$ is strictly decreasing and $H(0,\infty)=(0,\infty)$.
   Moreover, we have $H(\lambda r)\geq \lambda^{-2} H(r)$, for $r>0$, $\lambda>1$, hence
   \begin{equation}\label{eq:H-1}
     \lambda H^{-1}(s) \leq H^{-1}(\lambda^{-2} s),\quad s>0,\, \lambda>1.
   \end{equation}
   It follows from Proposition \ref{prop:1} that 
   $$
     C_8 H(1/r) \leq \Psi(r)\leq 2 H(1/r),\quad r>0,
   $$
   which yields
   \begin{equation}\label{eq:PsiH-1}
     \frac{1}{H^{-1}(s/2)} \leq \Psi^{-1}(s) \leq \frac{1}{H^{-1}(s/C_8)}, \quad s>0.
   \end{equation}
   Using (\ref{eq:PsiH-1}) and (\ref{eq:H-1}) we obtain
   $$
     \frac{h(at)}{h(t)} = \frac{\Psi^{-1}(1/t)}{\Psi^{-1}(1/at)}\leq \frac{H^{-1}(1/(2at))}{H^{-1}(1/(C_8t))}\leq 
     \sqrt{\frac{2a}{C_8}},
   $$
   for $a< C_8/2$, and the lemma follows.
\end{proof}

\begin{proof}[Proof of Theorem $\ref{th:p_est_below}$]
First we will prove that
there exist constants $c_1$, $c_2$, $c_3$ such that for every $a\in (0,1]$ we have
\begin{equation}\label{eq:TildaBelow}
  \tilde{p}^{h(at)}_t(y) \geq c_1 \left(h(t)\right)^{-d},  
\end{equation}
provided $|y|\leq c_2 e^{-c_3/a} h(t)$, $t\in T$.  

By symmetry of $\nu$ we have
$$
  \Fourier(\tilde{p}^{h(at)}_t)(\xi)
   \geq  |\Fourier(p_t)(\xi)|,\quad \xi\in\Rd,\, t\in T,
$$ and this and Lemma \ref{lm:Oppos} yield
\begin{eqnarray*}
  \tilde{p}^{h(at)}_t(0)
         & \geq & (2\pi)^{-d} \int e^{-t\Re(\Phi(\xi))}\, d\xi \\
         & \geq & c_4 \left(h(t)\right)^{-d}, \quad t\in T.
\end{eqnarray*}
By (\ref{eq:FTildePEstimate}) and (\ref{eq:Lm<Psi}) we have
$$
  |\Fourier(\tilde{p}^{h(at)}_t)(\xi)| \leq |\Fourier(p_t)(\xi)|e^{2t\nu(B(0,h(at))^c}
  \leq e^{-t\Re(\Phi(\xi))} e^{2M_0t\Psi(1/(h(at)))} = e^{-t\Re(\Phi(\xi))} e^{2M_0/a},
$$
and for every $j\in\{1,\dots,d\}$ by (\ref{eq:Fourier_int}) we get	
\begin{eqnarray*}
  \left|\frac{\partial \tilde{p}^{h(at)}_t }{\partial y_j}(y)\right|
  &   =  & \left|(2\pi)^{-d}\int
           (-i)\xi_j e^{-i\scalp{y}{\xi}}
           \Fourier(\tilde{p}^{h(at)}_t)(\xi)d\xi\right| \\
  & \leq & c_5 e^{2M_0/a} \left(\int_{|\xi|\leq (1/h(t))} e^{-t\Re(\Phi(\xi))}|\xi|\, d\xi 
              + \int_{|\xi|  > (1/h(t))} e^{-t\Re(\Phi(\xi))}|\xi|\, d\xi \right)\\
    & \leq & c_5 e^{2M_0/a} \left( c_6 \left(h(t)\right)^{-d-1} 
             +  \int e^{-t\Re(\Phi(\xi))}|\xi| \, d\xi \right) \\
    & \leq & c_7 e^{2M_0/a} \left(h(t)\right)^{-d-1}.
\end{eqnarray*}
It follows that
$$
  \tilde{p}^{h(at)}_t(y)\geq c_4\left(h(t)\right)^{-d}-dc_7 e^{2M_0/a} \left(h(t)\right)^{-d-1} |y|
  \geq \frac{1}{2}c_4 \left(h(t)\right)^{-d},
$$
provided  $|y|\leq \frac{c_4}{2dc_7}e^{-2M_0/a} h(t)$,
which clearly yields (\ref{eq:TildaBelow}).
  
Let $a\in(0,1)$ and $t\in T$.
For $r>0$, $|x|>r+h(at)$ by (\ref{eq:exp}) and (\ref{eq:Lm<Psi}) we get
$$
  \bar{P}^{h(at)}_t(B(x,r)) \geq e^{-M_0/a} t \bar{\nu}_{h(at)}(B(x,r)) = e^{-M_0/a} t \nu (B(x,r)).
$$
This, (\ref{eq:TildaBelow}) and (\ref{eq:nu_est_below}) for $x\in A$ yield 
\begin{eqnarray*}
  p_t(x+tb) 
  &  =    & \tilde{p}^{h(at)}_t * \bar{P}_t^{h(at)}(x) \\
  &  =   & \int \tilde{p}_t^{h(at)}(x-z)\bar{P}_t^{h(at)}(dz) \\
  & \geq & c_1 \int_{|z-x|<c_2e^{-c_3/a }h(t)} 
             \left(h(t)\right)^{-d} \bar{P}_t^{h(at)}(dz) \\
  &  =   & c_1 \left(h(t)\right)^{-d} \bar{P}_t^{h(at)}(B(x,c_2e^{-c_3/a}h(t))) \\
  & \geq & c_8 t\left(h(t)\right)^{-d+\gamma} f(|x|+c_2e^{-c_3/a}h(t) ), 
\end{eqnarray*}
for a constant $c_8=c_8(a)$, provided $|x|>h(at)+c_2e^{-c_3/a}h(t)$.
By Lemma \ref{lm:est_below_0} we have 
$p_t(x+tb)\geq C_{14} \left(h(t)\right)^{-d}$ for $|x|<C_6 h(t)$. Using Lemma \ref{lm:OdTomkaG} we choose $a\in(0,1)$ such that
$h(at)/h(t)+c_2e^{-c_3/a}\leq C_6$ and we obtain (\ref{eq:p_est_below2}) and (\ref{eq:p_est_below1}) follows from
Lemma \ref{lm:est_below_0}.
\end{proof}

\begin{lemat}\label{lm:derest}
  If $\nu(\Rd)=\infty$ and (\ref{eq:Fourier_derivatives}) holds for
  some $m\in\N_0$, then $\tilde{p}_t\in C^m_b(\Rd)$ 
  and for every $n\in\N_0$ such that $m\geq n$ and every $\beta\in\N^d_0$ 
  such that $|\beta|\leq m-n$ there exists a constant $C_{19}=C_{19}(m,n)$ such that
  $$
    |\partial^\beta_y \tilde{p}_t(y)| \leq 
    C_{19} \left[h(t)\right]^{-d-|\beta|}\left(1+|y|/h(t)\right)^{-n},
    \quad y\in\Rd,t\in T.
  $$
\end{lemat}

\begin{proof}
  The existence of the density $\tilde{p}_t\in C^m_b(\Rd)$ is a consequence of
  (\ref{eq:FTildePhest}), (\ref{eq:Fourier_derivatives}) and 
  \cite[Proposition 28.1]{Sato}. Similarly like in the proof of Lemma 
  \ref{lm:small_jumps_est} we consider $g_t(y)=[h(t)]^d\tilde{p}_t\left(h(t)y\right)$
  and the infinitely divisible distribution $\pi_t(dy)=g_t(y)\, dy$. It follows from
  (2.16) in \cite{SchSW12} that there exists a constant $c_1$ such that
  $$
    \int |y|^{n} \lambda_t(dy) \leq c_1,\quad t\in T,
  $$
  for every $n\geq 2$, where $\lambda_t$ is the L\'evy measure of $\pi_t$.
  Moreover using (\ref{eq:Fourier_derivatives}) and (\ref{eq:FTildePhest}) we get   
  \begin{eqnarray*}
    \int |\Fourier(\pi_t)(\xi)| |\xi|^m\, d\xi 
    &   =  & \int \left|\Fourier(\tilde{P}_t)
             \left(\xi/h(t)\right)\right||\xi|^m\, d\xi \\
    &   =  & \left[h(t)\right]^{d+m}\int \left|\Fourier(\tilde{P}_t)
             (\xi)\right| |\xi|^m\, d\xi \\
    & \leq & \left[h(t)\right]^{d+m}\int e^{2M_0}
             \exp\left[-t\,\Re\left(\Phi
             (\xi)\right)\right] |\xi|^m\, d\xi \leq M_8 e^{2M_0}.
  \end{eqnarray*}
  Using \cite[Proposition 2.1]{SchSW12} we obtain
  $$
    |\partial_y^\beta g_t(y)| \leq c_2 (1+|y|)^{-n},\quad y\in\Rd,  
  $$
  for $|\beta|+n\leq m$, and $c_2=c_2(m,n)$, and the lemma follows.  
\end{proof}

\begin{proof}[Proof of Theorem $\ref{th:Derivatives}$]
  The existence of the density $p_t\in C_b^m(\Rd)$ is a consequence of (\ref{eq:Fourier_derivatives}) and \cite{Sato}, Proposition
28.1, or \cite{P96}, Proposition 0.2.
  Using (\ref{eq:Fourier_derivatives}) we obtain
  \begin{eqnarray}\label{eq:der_global}
    \left|\partial^\beta_x p_t (x)\right|
    &   =  & \left|(2\pi)^{-d}\int
             (-i)^{|\beta|} \xi^\beta e^{-i\scalp{x}{\xi}}
             e^{-t\Phi(\xi)}\, d\xi\right| \nonumber \\
    & \leq & (2\pi)^{-d} \left(\int_{|\xi|\leq (1/h(t))} e^{-t\Re(\Phi(\xi))}|\xi|^{|\beta|}\, d\xi 
              + \int_{|\xi| > (1/h(t))} e^{-t\Re(\Phi(\xi))}|\xi|^{|\beta|}\, d\xi \right)
              \nonumber \\  
    & \leq & (2\pi)^{-d} \left(\int_{|\xi|\leq (1/h(t))} |\xi|^{|\beta|}\, d\xi 
              + \left[h(t)\right]^{m-|\beta|} \int_{|\xi| > (1/h(t))} 
              e^{-t\Re(\Phi(\xi))}|\xi|^m \, d\xi \right) \nonumber \\
    & \leq & c_1 \left(h(t)\right)^{-d-|\beta|},
  \end{eqnarray}
  for $x\in\Rd$ and $t\in T$. It follows from Lemma \ref{lm:derest}, Corollary \ref{cor:nub_est} and (\ref{eq:bar_P_schonwieder}) that
  \begin{eqnarray*}
    \left|\partial_x^\beta\left(\tilde{p}_t \ast \bar{P}_t \right)(x)\right|
    &   =  & \left| \int \partial_x^\beta \tilde{p}_t(x-y) \bar{P}_t(dy) \right| \\
    & \leq & \int \left| \partial_x^\beta \tilde{p}_t(x-y) \right| \bar{P}_t(dy) \\
    & \leq & C_{19}\left[h(t)\right]^{-d-|\beta|} \int (1+|x-y|/h(t))^{-n} \bar{P}_t(dy) \\
    &   =  & C_{19}\left[h(t)\right]^{-d-|\beta|} \int \int_0^{(1+|x-y|/h(t))^{-n}} \,ds\, \bar{P}_t(dy)\\
    &   =  & C_{19} [h(t)]^{-d-|\beta|} \int_0^1 \int 
             \indyk{\{y\in\Rd:\: (1+|x-y|/h(t))^{-n}>s\}} \, \bar{P}_t(dy) ds \\
    &   =  & C_{19} [h(t)]^{-d-|\beta|} \int_0^1 \bar{P}_t\left(B(x,h(t)(s^{-\frac{1}{n}}-1))\right) ds \\
    & \leq & c_2 [h(t)]^{-d-|\beta|} \left(\int_{(1+\frac{|x|}{2h(t)})^{-n}}^1 t f\left(|x|/4\right) 
             \left(h(t)(s^{-\frac{1}{n}}-1)\right)^\gamma\, ds 
             + \int_0^{(1+\frac{|x|}{2h(t)})^{-n}}\, ds\right) \\
    & \leq & c_2 [h(t)]^{-d-|\beta|} \left( t[h(t)]^{\gamma} f\left(|x|/4\right) \int_0^1  
             s^{-\gamma/n}\, ds + \left(1+\frac{|x|}{2h(t)}\right)^{-n} \right) \\
    &  =   & c_3 [h(t)]^{-d-|\beta|} \left( t[h(t)]^{\gamma} f\left(|x|/4\right) 
             + \left(1+\frac{|x|}{2h(t)}\right)^{-n} \right),
  \end{eqnarray*}
  for $x\in\Rd$, $t\in T$, and this and (\ref{eq:der_global}) yield (\ref{eq:der_est}).
\end{proof}


\section{Examples}\label{Examples}

In what follows we assume that
$$
  \nu(A) \approx \int_\sfera \int_0^\infty \indyk{A}(s\theta) Q(s)\,  ds\mu(d\theta),
$$
for nondegenerate measure $\mu$ and a nonincreasing function $Q$. We assume also
that $\mu$ is a $\gamma-1$-measure on $\sfera$ for some $\gamma\in [1,d]$, i.e.
\begin{equation}\label{eq:mu_gamma_measure}
  \mu(\sfera\cap B(\theta,\rho)) \leq c \rho^{\gamma -1},\quad \theta\in\sfera,\,\rho>0.
\end{equation}

It is easy to check that
$$
  \nu(A) \leq c Q(\delta(A))(\delta(A))^{1-\gamma}[\diam(A)]^{\gamma},\quad A\in{\mathcal{B}}(\Rd ),
$$
and so the assumption (\ref{eq:nu_estim}) is satisfied with $f(s)=s^{1-\gamma}Q(s)$. Furthermore, it 
follows from (\ref{eq:Lm<Psi})
that (\ref{eq:tech_assumpt}) holds for every $Q$ such that
\begin{equation}\label{eq:doubling}
  Q(s)\leq c Q(2s), \quad s > 0.
\end{equation}

In the following theorem we obtain upper estimates for a specific class of jump processes. 
For simplification we include here only symmetric case and $b=0$.

\begin{twierdzenie}\label{th:ExE}
  Let $\alpha\in(0,2]$, $\kappa>0$, $\alpha>\kappa\beta > \alpha -2$, and $\beta>1$ if $\alpha=2$. If the L\'evy measure $\nu$ satisfies 
  \begin{equation}\label{eq:stablog}
    \nu(A) \approx \int_\sfera \int_0^\infty \indyk{A}(s\theta) 
                 s^{-1-\alpha}\left[\log\left(1+s^{-\kappa}\right)\right]^{-\beta} \,  ds\mu(d\theta),
  \end{equation} 
  is symmetric, i.e. $\nu(-A)=\nu(A)$, $b=0$, $\mu$ is nondegenerate and there exists a constant $\gamma\in [1,d]$ such 
  that (\ref{eq:mu_gamma_measure}) holds then the measures $P_t$ are absolutely continuous with respect 
  to the Lebesgue measure and their densities $p_t$ satisfy the following estimates.
\begin{enumerate}
    \item Short time estimates:
    \begin{enumerate}
      \item for $\alpha\in (0,2)$ there exists a constant $C_{20}$ such that for every $t\in (0,1)$ and $x\in\Rd$, we have
        $$
          p_t(x) 
          \leq  C_{20} t^{-d/\alpha} \left(\log\left( 1+1/t \right)\right)^{d\beta/\alpha}
          \min\left\{ 1, \frac{t^{1+\gamma/\alpha}\left[\log\left(1+|x|^{-\kappa}\right)\right]^{-\beta}}{ \left(\log\left( 1+1/t \right)\right)^{\gamma\beta/\alpha}
                 |x|^{\gamma+\alpha}}\right\}.
        $$
      \item for $\alpha=2$ there exist constants $C_{21},C_{22},C_{23}$ such that for every $t\in (0,1)$ and $x\in\Rd$, we have
        \begin{eqnarray}\label{e:1}
          p_t(x) 
          & \leq  & C_{21} t^{-d/2} \left(\log\left( 1+1/t \right)\right)^{d(\beta-1)/2} \nonumber \\          
          &       &  \times  \min \left\{ 1, \frac{t^{1+\gamma/2}
                        \left[\log\left(1+|x|^{-\kappa}\right)\right]^{-\beta}} 
                        {\left(\log\left( 1+1/t \right)\right)^{\gamma(\beta-1)/2}
                 |x|^{\gamma+2}}+\, e^{ \frac{-C_{22}|x|}{h(t)}
                  \log(1+\frac{C_{23}|x|}{h(t)})}\right\}. 
        \end{eqnarray}
    \end{enumerate}
    \item Large time estimates: for $\alpha\in (0,2]$ there exists a constant $C_{24}$ such that for every $t >1$ 
    and $x\in\Rd$, we have
      $$
          p_t(x) 
          \leq  C_{24} t^{-d/(\alpha-\kappa\beta)}  \min\left\{ 1, 
                 t^{1+\gamma/(\alpha-\kappa\beta)}
                 |x|^{-\gamma-\alpha}\left[\log\left(1+|x|^{-\kappa}\right)\right]^{-\beta} \right\}.
      $$
  \end{enumerate}
\end{twierdzenie}
\begin{proof}
Let 
$$
  Q(s)=s^{-1-\alpha}\left[\log(1+s^{-\kappa})\right]^{-\beta},\quad s\in (0,\infty).
$$ 
The function $Q$ is decreasing, satisfies (\ref{eq:doubling})
and $\int_0^\infty (1\wedge s^2) Q(s)\, ds < \infty$. Furthermore for $r\in (0,1)$ we have
\begin{eqnarray*}
  \int_0^r s^2 Q(s)\, ds 
  & \approx & \int_0^r s^{1-\alpha} \left[\log(2s^{-\kappa})\right]^{-\beta} \, ds\\
  &    =    & \kappa^{-\beta} 2^{(2-\alpha)/\kappa} \int_{\log(2^{1/\kappa}/r)}^\infty e^{-u(2-\alpha)}u^{-\beta}\, du \\
  & \approx & \left\{
  \begin{array}{lcl}
    r^{2-\alpha} \left[\log(1+\frac{1}{r})\right]^{-\beta} & \mbox{  for  } & \alpha\in (0,2),\\
    \left[\log(1+\frac{1}{r})\right]^{-\beta+1} & \mbox{  for  } & \alpha=2,
  \end{array}\right.
\end{eqnarray*}
and for $r >1$ we get
$$  \int_0^r s^2 Q(s)\, ds 
   \approx  \int_0^1 s^2 Q(s)\, ds + \int_1^r s^{1-\alpha} s^{\kappa\beta} \, ds
   \approx  r^{2-\alpha+\kappa\beta}.
$$
Using Corollary  \ref{l:RePhi<} and \ref{l:RePhi>} (with decreasing function $g(r)=r^{-(\kappa\beta\vee 0)}\left[\log\left(1+r^{-\kappa}\right)\right]^{-\beta}$ for $\alpha\in(0,2)$ and 
$g(r)=r^{-\kappa\beta }\left[\log\left(1+r^{-\kappa\beta/(\beta-1)}\right)\right]^{1-\beta}$ for $\alpha=2$)  we obtain
$$
  \Phi(\xi) \approx |\xi|^\alpha \left[\log\left(1+|\xi|^\kappa\right)\right]^{-\beta},
$$
for $\alpha\in (0,2)$,
and 
$$
  \Phi(\xi) \approx |\xi|^2 \left[\log\left(1+|\xi|^{\kappa\beta/(\beta-1)}\right)\right]^{1-\beta}
$$
for $\alpha=2$.

For $s>0$, set $F_\alpha(s)=s^\alpha (\log(1+s^\kappa))^{-\beta}$ for $\alpha\in (0,2)$ and 
$F_2(s)=s^2 \left[\log\left(1+s^{\kappa\beta/(\beta-1)}\right)\right]^{1-\beta}$.
The functions $F_\alpha$ are increasing for every $\alpha$. We let 
$g_\alpha(r)=\big(r (\log(1+r))^{\beta}\big)^{1/\alpha}$ for $\alpha\in (0,2)$ and 
$g_2(r)=\left( r \left[\log(1+r)\right]^{\beta-1}\right)^{1/2}$. Then there exists $r_0=r_0(\alpha,\kappa,\beta)$ such that for $r>r_0$ and $\alpha\in (0,2)$ we have 
\begin{align*}
    F_\alpha \big(g_\alpha(r)\big)
    &= r\, \big(\log(1+r)\big)^{\beta} \Big[\log\Big(1+\big(r(\log(1+r))^{\beta}\big)^{\kappa/\alpha} \Big)\Big]^{-\beta}\\
    &\approx  r\, (\log r)^{\beta} \left(\log r +\beta\log\log r\right)^{-\beta}\\
    &= r \left[\frac{\log r +\beta\log\log r}{\log r}\right]^{-\beta}\\
    &\approx r.
\end{align*}
Similarly $F_2(g_2(r))\approx r$ for sufficiently large $r$.
This shows that $F_\alpha^{-1}(r)\approx g_\alpha(r)$ for $r>r_0$. For $r<r_1=r_1(\alpha,\kappa,\beta)$ we have $F_\alpha(r)\approx r^{\alpha-\kappa\beta}$ and
$F_\alpha^{-1}(r)\approx r^{1/(\alpha-\kappa\beta)}$. It yields
$$
  h(t)=\frac{1}{\Psi^{-1}\left(\frac{1}{t}\right)}\approx t^{\frac{1}{\alpha-\kappa\beta}},\quad t\geq 1,
$$
and 
$$
  h(t) \approx t^{1/\alpha} \left[\log\left(1+\frac{1}{t}\right)\right]^{\frac{-\beta}{\alpha}},\quad t\in (0,1),
$$
for $\alpha\in (0,2)$, and
$$
  h(t) \approx t^{1/2} \left[\log\left(1+\frac{1}{t}\right)\right]^{\frac{1-\beta}{2}},\quad t\in (0,1),
$$
for $\alpha=2$.
Moreover the assumptions of Lemma \ref{lm:OdRene} and Theorem \ref{th:main} are satisfied with $T=(0,\infty)$.
The estimate given in Theorem \ref{th:main} holds, i.e.
\begin{eqnarray}\label{ep1}
  p_t(x) 
  & \leq & C_1 \left(h(t)\right)^{-d} \min\left\{ 1, t\left[h(t)\right]^{\gamma}
           f\left(|x|/4\right)
          + \, e^{-C_2 \frac{|x|}{h(t)}\log\left(1+\frac{C_3|x|}{h(t)}\right)}
            \right\},\\
   &     & x\in\Rd,\, t\in (0,\infty), \nonumber
\end{eqnarray}
for $f(s)=s^{1-\gamma}Q(s)$.
We have 
$$
  t=\frac{1}{\Psi(1/h(t))} \approx h(t)^\alpha \left[\log(1+h(t)^{-\kappa})\right]^\beta,
$$
for $\alpha\in (0,2)$ and
$$
  t=\frac{1}{\Psi(1/h(t))} \approx h(t)^2 \left[\log(1+h(t)^{\frac{\kappa\beta}{1-\beta}})\right]^{\beta-1},
$$
for $\alpha=2$. Let $g(t,|x|)=th(t)^\gamma f(|x|/4)$. For $\alpha\in (0,2)$ we obtain
\begin{eqnarray*}
  g(t,|x|) 
  & \approx & h(t)^{\alpha+\gamma} \left[\log(1+h(t)^{-\kappa})\right]^\beta |x|^{-\alpha-\gamma} \left[\log(1+|x|^{-\kappa})\right]^{-\beta} \\
  &   =     & \left[\frac{|x|}{h(t)}\right]^{-\alpha-\gamma}\left[\frac{\log(1+h(t)^{-\kappa})}{\log(1+|x|^{-\kappa})}\right]^{\beta}.
\end{eqnarray*}
Using the fact that
$$
  \frac{u}{v}\wedge 1 \leq \frac{\log(1+u)}{\log(1+v)} \leq \frac{u}{v} \vee 1,\quad u,v>0,
$$
we get
\begin{equation}\label{ep2}
  g(t,|x|) \geq c_1 e^{-C_2 \frac{|x|}{h(t)}\log\left(1+\frac{C_3|x|}{h(t)}\right)},
\end{equation}
for some constant $c_1$.
Similarly for $\alpha=2$ we have 
\begin{eqnarray*}
  g(t,|x|) 
  & \approx & \left[\frac{|x|}{h(t)}\right]^{-2-\gamma}\left[\frac{\log(1+h(t)^{-\kappa})}{\log(1+|x|^{-\kappa})}\right]^{\beta}
               \frac{\left[\log(1+h(t)^{\frac{\kappa\beta}{1-\beta}})\right]^{\beta-1}}{\left[\log(1+h(t)^{-\kappa})\right]^\beta}.
\end{eqnarray*}
Let
$$
  A(t)=\frac{\left[\log(1+h(t)^{\frac{\kappa\beta}{1-\beta}})\right]^{\beta-1}}{\left[\log(1+h(t)^{-\kappa})\right]^\beta}.
$$
We have $c_2^{-1}\leq A(t)\leq c_2$, for some constant $c_2$ and $t\geq 1$ (but $A(t)\to 0$ for $t\to 0$).
Therefore
\begin{equation}\label{ep3}
  g(t,|x|) \geq c_3 e^{-C_2 \frac{|x|}{h(t)}\log\left(1+\frac{C_3|x|}{h(t)}\right)},
\end{equation}
for $t\geq 1$.
The theorem follows from (\ref{ep1}), (\ref{ep2}) and (\ref{ep3}).
\end{proof}

The case of $\nu$ satisfying locally (\ref{eq:stablog}) with $\alpha=2$, $\beta\in (1,2]$ and $\gamma=d$ for $t<1$ and $|x|<1$ was investigated also in \cite{Mimica1}. Theorem 1.1 in \cite{Mimica1}
contains the estimate
$$
  p_t(x) \leq  c_1        
        \min \left\{ t^{-d/2} \left(\log \frac{2}{t} \right)^{d(\beta-1)/2}, \frac{t
                        } 
                        {|x|^{d+2} \left(\log\frac{2}{|x|}\right)^{\beta-1}}\right\}, \quad |x|<1,t<1,
$$
and here (\ref{e:1}) with $\gamma=d$ yields
$$ p_t(x) \leq   c_2 \min \left\{ t^{-d/2} \left(\log \frac{2}{t} \right)^{d(\beta-1)/2},\frac{t
                        }{|x|^{d+2} \left(\log\frac{2}{|x|}\right)^{\beta}} +\, h(t)^{-d}e^{ \frac{-C_{22}|x|}{h(t)}
                  \log(1+\frac{C_{23}|x|}{h(t)})}\right\},
$$
for $|x|<1$ and $t<1$. We note that there exists a constans $c_3$
such that
$$
  \frac{t}{|x|^{d+2} \left(\log\frac{2}{|x|}\right)^{\beta}} +\, h(t)^{-d} e^{ \frac{-C_{22}|x|}{h(t)}
                  \log(1+\frac{C_{23}|x|}{h(t)})} \leq c_3 \frac{t}{|x|^{d+2} \left(\log\frac{2}{|x|}\right)^{\beta-1}},
$$
for $|x|<1$, $t<1$, which can be shown similarly as (\ref{ep3}) in the proof of Theorem \ref{th:ExE}, and so 
(\ref{e:1}) improves the results of \cite{Mimica1}. Exact estimate
in this case is still an open question.

Using Theorem \ref{th:p_est_below}  we can obtain estimates from below. For example, if $\nu$
satisfies the assumptions of Theorem \ref{th:ExE} and additionally 
for some finite set $D_0=\{\theta_1,\theta_2,...,\theta_n\}\subset\sfera$ and a positive constant $c_0$ we have
$$
  \mu(\{\theta_k\}) \geq c_0,\quad \theta_k\in D_0,\, k=1,2,...,n,
$$
 then
 $$
          p_t(x) 
          \geq  c_4 t^{-d/\alpha} \left(\log\left( 1+1/t \right)\right)^{d\beta/\alpha}
          \min\left\{ 1, 
          \frac{t^{1+1/\alpha}\left[\log\left(1+|x|^{-\kappa}\right)\right]^{-\beta}}{ \left(\log\left( 1+1/t \right)\right)^{\beta/\alpha}
                 |x|^{1+\alpha}}\right\}, \quad t\in (0,1),
 $$
 for $\alpha\in (0,2)$ and $x\in D=\{x\in\Rd:\: x=r\theta,\,r>0,\theta\in D_0\}$,
 and        
$$
          p_t(x) 
           \geq   c_5 t^{-d/2} \left(\log\left( 1+1/t \right)\right)^{d(\beta-1)/2}          
                   \min \left\{ 1, \frac{t^{3/2}
                        \left[\log\left(1+|x|^{-\kappa}\right)\right]^{-\beta}} 
                        {\left(\log\left( 1+1/t \right)\right)^{(\beta-1)/2}
                 |x|^{3}}\right\},\quad t\in (0,1),
$$
for $\alpha=2$ and $x\in D$, and
$$
          p_t(x) 
          \geq  c_6 t^{-d/(\alpha-\kappa\beta)}  \min\left\{ 1, 
                 t^{1+1/(\alpha-\kappa\beta)}
                 |x|^{-1-\alpha}\left[\log\left(1+|x|^{-\kappa}\right)\right]^{-\beta} \right\}, \quad t\geq 1,
      $$
for $\alpha \in (0,2]$, $x\in D$.
        
In the following corollary we consider the case of $\alpha\in(0,2)$, $\gamma=d$ and give both side estimates of the densities and
estimates of their derivatives. We omit the proof which is
a verification of the assumption of Theorems \ref{th:main}, \ref{th:p_est_below} and
\ref{th:Derivatives} analogous to the proof of Theorem \ref{th:ExE}.

\begin{wniosek}
  Let $\alpha\in(0,2)$, $\kappa>0$, $\alpha>\kappa\beta > \alpha -2$.
  If the L\'evy measure $\nu(dy)=g(y)dy$ satisfies 
  $$
    g(y) \approx |y|^{-d-\alpha}\left[\log\left(1+|y|^{-\kappa}\right)\right]^{-\beta},
  $$
  and is symmetric, i.e. $g(-y)=g(y)$, then the measures $P_t$ are absolutely continuous with respect 
  to the Lebesgue measure and their densities $p_t$ satisfy
        $$
          p_t(x) 
          \approx  \min\left\{ t^{-d/\alpha} \left(\log\left( 1+1/t \right)\right)^{d\beta/\alpha},                   t|x|^{-d-\alpha}\left[\log\left(1+|x|^{-\kappa}\right)\right]^{-\beta}\right\},
          \quad t\in (0,1),x\in\Rd,
        $$
  and
      $$
          p_t(x) 
          \approx \min\left\{ t^{-d/(\alpha-\kappa\beta)}, 
                 t|x|^{-d-\alpha}\left[\log\left(1+|x|^{-\kappa}\right)\right]^{-\beta} \right\},\quad
                 t\geq 1,x\in\Rd.
      $$
  Furthermore, for every $\eta\in\N^d_0$ there exist constants $C_{25},C_{26}$ such that for
  $t\in (0,1),x\in\Rd,$ we have
      $$
          |\partial^\eta_x p_t(x)|
          \leq C_{25} t^{(-d-|\eta|)/\alpha} \left(\log\left( 1+1/t \right)\right)^{(d+|\eta|)\beta/\alpha}
           \min\left\{ 1, \frac{t^{1+d/\alpha}\left[\log\left(1+|x|^{-\kappa}\right)\right]^{-\beta}}
           { |x|^{d+\alpha} \left(\log\left( 1+1/t \right)\right)^{d\beta/\alpha}}\right\},
        $$
  and for $t\geq 1,x\in\Rd,$ we have
      $$
         |\partial^\eta_x p_t(x)|
          \leq C_{26} t^{(-d-|\eta|)/(\alpha-\kappa\beta)} 
          \min\left\{ 1, 
                 t^{1+d/(\alpha-\kappa\beta)}
                 |x|^{-d-\alpha}\left[\log\left(1+|x|^{-\kappa}\right)\right]^{-\beta} \right\}.
      $$
\end{wniosek}

If $Q(s) \approx s^{-1-\alpha}q(s)\phi(s)$ where $\alpha\in (0,2)$, and the functions $s^{-1-\alpha}q(s)$, $\phi(s)$ are
nonincreasing and positive on $[0,\infty)$, $q$ and $\phi$ are bounded and satisfy
$$
  q(s)\leq c q(2s),\quad \phi(s_1)\phi(s_2)\leq c\phi(s_1+s_2),
$$
for every $s,s_1,s_2>0$ and some constant $c$, and if
$$
  \int_0^\infty s^{1-\alpha} q(s) \frac{\phi(s)}{\phi(s/2)}\, ds < \infty,
$$
then it can be checked that (\ref{eq:tech_assumpt}) holds. Such examples of L\'evy measures were investigated in \cite{S11, S10,KSz} so we
do not repeat here detailed estimates of their densities but we give below some estimates of the derivatives which follow from Theorem \ref{th:Derivatives}.

\begin{wniosek}
Let $m\geq 0$, $\beta \in (0,1]$, $\alpha \in (0,2)$, $\kappa\leq 1+\alpha$, and $\kappa<\alpha$ if $m = 0$. If the L\'evy measure $\nu$ satisfies
$$
\nu(A) \approx \int_{\sfera} \int_0^{\infty} \indyk{A} (s \theta) s^{-1-\alpha} (1+s)^{\kappa} e^{-ms^{\beta}} ds \mu(d \theta),
$$
is symmetric, i.e., $\nu(A)=\nu(-A)$, $b=0$, $\mu$ is nondegenerate and fulfills \eqref{eq:mu_gamma_measure} with $\gamma\in [1,d]$, then the measures 
$P_t$ are absolutely continuous with respect to the Lebesgue measure and their densities $p_t\in C^\infty_b(\Rd)$ satisfy
\begin{eqnarray*}
  |\partial^\eta_x p_t(x)| 
  & \leq & C_{27} \left(h(t)\right)^{-d-|\eta|} \min\left\{ 1,\,\, 
           \frac{t\left[h(t)\right]^{\gamma}(1+|x|)^\kappa}{|x|^{\gamma+\alpha}} e^{-m(|x|/4)^{\beta}}
          + \, \left(1+\frac{|x|}{h(t)}\right)^{-n}
            \right\},\nonumber \\
  &      & x\in\Rd,\, t >0.
\end{eqnarray*}
for every $|\eta|\in\N_0^d$, 
every $n\in\N$, $n>\gamma$ and a constant $C_{27}=C_{27}(|\eta|,n)$, where
$$ 
  h(t)\approx t^{1/\alpha},\quad \mbox{for} \quad t<1,
$$
and 
$$
  h(t)\approx \left\{
  \begin{array}{lcl}
    t^{1/2} & \mbox{ for }  & m>0,\,\kappa \leq 1 + \alpha, \\
    t^{1/2} & \mbox{ for }  & m=0,\,\kappa<\alpha-2,\\
    (t\log (1+t))^{1/2} & \mbox{ for }  &\, m=0,\kappa=\alpha-2,\\
    t^{1/(\alpha-\kappa)} & \mbox{ for }  &\, m=0,\alpha-2<\kappa<\alpha.
  \end{array}
  \right.
$$
for $t>1$.
\end{wniosek}

We consider in the last example the discrete L\'evy measure
$$
  \nu(dy) = \sum_{i=1}^d \sum_{n=-\infty}^\infty 2^{n\beta}\left(\delta_{2^{-n\kappa}e_i}(dy)+\delta_{-2^{-n\kappa}e_i}(dy) \right),
$$
where $0<\beta<2\kappa$ and $\{e_i\}_{i=1}^d$ is the standard basis in $\Rd$. Using Proposition \ref{prop:1} we
easily get $\Re\left(\Phi(\xi)\right)\approx |\xi|^{\beta/\kappa}$, $h(t)\approx t^{\kappa/\beta}$ in this case. It follows 
from Lemma \ref{lm:OdRene} that (\ref{eq:Fourier_int}) and (\ref{eq:Fourier_derivatives}) are satisfied with $T=(0,\infty)$ and 
it is also easy to check that (\ref{eq:nu_estim}) and (\ref{eq:tech_assumpt}) hold with
 $\gamma=0$ and $f(s)=s^{-\beta/\kappa}$. Therefore for the corresponding semigroup (we let $b=0$) we obtain the following estimate of
the density
$$
  p_t(x) \leq c_1 t^{-d\kappa/\beta} \min\left\{1,t|x|^{-\beta/\kappa}\right\},\quad t>0,x\in\Rd,
$$ 
and their derivatives
$$
  |\partial_x^\eta p_t(x)| \leq c_2 t^{-(d+|\eta|)\kappa/\beta} \min\left\{1,t|x|^{-\beta/\kappa}\right\},\quad t>0,x\in\Rd,
$$
for every $\eta\in\N_0^d$ with $c_2=c_2(\eta)$. Using Theorem \ref{th:p_est_below} with $A=\supp\nu$ we obtain the lower estimate
$$
  p_t(x) \geq c_3 t^{-d\kappa/\beta} \min\left\{1,t|x|^{-\beta/\kappa}\right\},\quad t>0,x\in \{2^{-n\kappa}e_i,-2^{-n\kappa}e_i:\: n\in\Z,i=1,...,d\}.
$$
The estimates in this case for $d=1$ were obtained previously in \cite{KnopKul2} (see Example 4.2).

\end{document}